\documentclass[12pt]{article}
\usepackage{amsmath,amsfonts,amssymb,amsthm}
\newcommand{\N}{\mathbb N}
\newcommand{\Z}{\mathbb Z}
\newcommand{\twosum}[2]{\sum_{\substack{#1\\#2}}}
\newtheorem{defn}{Definition}[section]
\newtheorem{thm}[defn]{Theorem}
\newtheorem{lem}[defn]{Lemma}

\title{Average Bounds for Kloosterman Sums Over Primes}
\author{A.J. Irving\\
Mathematical Institute, Oxford}
\date{}
\begin{document}
\maketitle

\section{Introduction}

In this paper we consider bounds for sums of the form 
$$S_q(a;x)=\twosum{p\sim x}{(p,q)=1}e(\frac{a\overline p}{q})$$
where $p$ runs over primes and $p\sim x$ denotes the inequality $x\leq p<2x$.  These sums may be bounded trivially by $x$. If $(a,q)=1$ then we conjecture that for any $\epsilon>0$ a bound of 
$$S_q(a;x)\ll_\epsilon x^{\frac12+\epsilon}q^\epsilon$$   
should be true.  This conjecture, however, seems to be far out of reach of current methods.

A bound for $S_q(a;x)$ is given by Garaev, \cite{garaev}, in the case that $q$ is prime.  He shows that for $x<q$ we have, for any $\epsilon>0$,  
$$\max_{(a,q)=1}|S_q(a;x)|\ll_\epsilon (x^{\frac{15}{16}}+x^{\frac23}q^{\frac14})q^\epsilon.$$
This gives us a nontrivial estimate for the sum provided that $x\geq
q^{\frac34+\delta}$ for some $\delta>0$.  For $q\geq x\geq
q^{\frac{16}{17}+\delta}$  Garaev uses this bound to prove an
asymptotic for the number of prime solutions, $p_1,p_2,p_3$ with $p_i\in [0,x]$, to the congruence 
$$p_1(p_2+p_3)\equiv \lambda\pmod q.$$

Fouvry and Shparlinski, \cite{foushpar}, generalise Garaev's  bound to arbitrary $q$ and the larger range $q^{\frac34}\leq x\leq q^{\frac43}$.  They also show that if we average over $q$ then a sharper bound is possible.  Specifically, their Theorem 5 states that if $Q^{\frac32}\geq x\geq 2$ then for every $\epsilon>0$ we have 
\begin{equation}\label{foushparaverage}
\sum_{q\sim Q}\max_{(a,q)=1}|S_q(a;x)|\ll_\epsilon (Q^{\frac{13}{10}}x^{\frac35}+Q^{\frac{13}{12}}x^{\frac56})Q^\epsilon.
\end{equation}
This bound is  nontrivial when $x\geq Q^{\frac34+\delta}$.  Fouvry and Shparlinski use their estimates to study multiplicative properties of the set 
$$\{p_1p_2+p_1p_3+p_2p_3:p_i\sim x,p_i\text{ prime }\}.$$
They show, for example, that for $x$ sufficiently large this set contains numbers with a prime factor exceeding $x^{1.10028}$.  Baker, \cite[Theorem 2]{baker}, has recently improved the bound (\ref{foushparaverage}) in the range $Q^{\frac12}\leq x\leq 2Q$. His result is 
 \begin{equation}\label{bakeraverage}
\sum_{q\sim Q}\max_{(a,q)=1}|S_q(a;x)|\ll_\epsilon (Q^{\frac{11}{10}}x^{\frac45}+Qx^{\frac{11}{12}})Q^\epsilon.
\end{equation}
This  is nontrivial for $Q\geq x^{\frac12+\delta}$ and sharper than (\ref{foushparaverage}) when $x\leq Q^{1-\delta}$.  Baker applies this bound to the same ternary form problem as Fouvry and Shparlinski; combining it with a variant on the sieve argument he improves $1.10028$ to $1.1673$.

We are motivated by a new application of these sums to Diophantine
approximation.  For this application we need only consider average
bounds.  We will show that by generalising the arguments from
\cite{foushpar} it is possible to obtain a sharper estimate than that
given in (\ref{foushparaverage}).
\begin{thm}\label{thm1}
For any $\epsilon>0$ we have 
$$\sum_{q\sim Q}\max_{(a,q)=1}|S_q(a;x)|\ll_\epsilon (Q^{\frac54}x^{\frac58}+Qx^{\frac{9}{10}}+Q^{\frac76}x^{\frac{13}{18}})Q^\epsilon$$
for $Q^{\frac32}\geq x\geq Q^{\frac23}$.  
\end{thm}

This gives us a nontrivial result for $x\geq Q^{\frac23+\delta}$.  The
proof uses similar methods to those of Fouvry and Shparlinski.
However we introduce higher moments into
one of their estimates.  This results in a problem of counting solutions to a congruence with a larger number of variables; one which we  can solve with a sharp bound when we average over $q$.

Using this theorem we give a further improvement of the exponent in the ternary form problem.  Let $P^+(n)$ denote the largest prime factor of $n$.

\begin{thm}\label{ternary}
Let $\theta_1=1.188\ldots$ be the unique root of the equation 
$$42\theta-65+38\log\left(\frac{21\theta-19}{4}\right)=0.$$
Then, for any $\theta<\theta_1$,
$$\#\{p_1,p_2,p_3:p_i\sim x,p_i\text{ prime},P^+(p_1p_2+p_1p_3+p_2p_3)>x^\theta\}\gg_\theta \frac{x^3}{(\log x)^3}.$$
\end{thm}

In some applications of Theorem \ref{thm1} the maximum over $a$ is not necessary.  We
therefore prove a bound when $a$ is constant, which is stronger provided that $a$ is not too large.

\begin{thm}\label{thm2}
For any integer $a>0$ and any $\epsilon>0$ we have 
$$\sum_{q\sim Q}|S_q(a;x)|\ll_\epsilon (1+\frac{a}{xQ})^{\frac12}(Q^{\frac12}x^{\frac{11}{8}}+Q^{\frac76}x^{\frac23})(aQ)^\epsilon$$ 
for $Q^{\frac43}\geq x\geq Q^{\frac12}$.
\end{thm}
This is nontrivial for $Q^{\frac12+\delta}\leq 
x\leq Q^{\frac43-\delta}$.  The proof exploits the fact that, since there is no maximum over $a$, we can reorder summations to give an inner sum over $q\sim Q$.  This is a longer range than those arising in the proof of Theorem \ref{thm1}.  After inverting the Kloosterman fractions in such sums we reach a situation in which the Weil estimate for short Kloosterman sums can be used.

The sums in this last theorem are essentially bilinear forms with Kloosterman fractions, which were studied for arbitrary coefficients by Duke, Friedlander and Iwaniec in \cite{dfi}.  In the case that one of the coefficients is the indicator function of the primes then our theorem does better than the general result of \cite{dfi}, provided that  $x$ and $Q$ are sufficiently close in size.

We are most interested in the situation when $x\asymp Q$, that is
$Q\ll x\ll Q$, as this is the case in our application to Diophantine
approximation.  For this reason we have given theorems which, given
our current ideas, are as sharp as possible in this case.  For $x$
sufficiently different in size to $Q$ it is possible to improve the
above theorems.  In order to compare the various
results we note
that when $x\asymp Q$ we have the following bounds, valid for any
$\epsilon>0$. 

\begin{enumerate}
\item Using Fouvry and Shparlinski's bound (\ref{foushparaverage}) or Baker's  (\ref{bakeraverage}), we get  
$$\sum_{q\sim Q}\max_{(a,q)=1}|S_q(a;x)|\ll_\epsilon Q^{\frac{23}{12}+\epsilon}.$$

\item Theorem \ref{thm1} improves this to 
$$\sum_{q\sim Q}\max_{(a,q)=1}|S_q(a;x)|\ll_\epsilon Q^{\frac{19}{10}+\epsilon}.$$

\item If $0<a\ll Q^2$ then using Theorem 2 from Duke, Friedlander and Iwaniec, \cite{dfi}, we get a bound 
$$\sum_{q\sim Q}|S_q(a;x)|\ll_\epsilon Q^{\frac{95}{48}+\epsilon}.$$

\item If $0<a\ll Q^2$ then Theorem \ref{thm2} gives a bound 
$$\sum_{q\sim Q}|S_q(a;x)|\ll_\epsilon Q^{\frac{15}{8}+\epsilon}.$$
\end{enumerate}

These bounds should be compared with  the trivial bound of  $Q^2$ and the conjectured best bound of $Q^{\frac32+\epsilon}$.

\subsection*{Acknowledgements}

This work was completed as part of my DPhil, for which I was funded by EPSRC grant EP/P505666/1. I am very grateful to the
EPSRC for funding me and to my supervisor, Roger Heath-Brown, for all
his help and advice. 

\section{Lemmas}

We require the following estimate for short Kloosterman sums coming from the Weil bound.

\begin{lem}\label{weil}
For integers $a$ and $q$ with $q>1$, and reals $Y<Z$ we have, for any $\epsilon>0$, that  
$$\twosum{Y<n\leq Z}{(n,q)=1}e(\frac{a\overline n}{q})\ll_\epsilon ((a,q)(\frac{Z-Y}{q}+1)+q^{\frac12})q^\epsilon.$$
\end{lem}
\begin{proof}
This is a slight weakening of Lemma 1 from Fouvry and Shparlinski, \cite{foushpar}.  It follows immediately on inserting the estimate 
$$n^{1-\epsilon}\ll \phi(n)\ll n$$
as well as the standard bound for the divisor function, $\tau$, into that lemma.
\end{proof}

We also require the following estimate for the number of solutions to a certain Diophantine equation.

\begin{lem}\label{cayley}
Let $k\in \N$ and $\epsilon>0$ be fixed.  For any $N\geq 0$ we have 
$$\#\{(n_1,\ldots,n_{2k})\in \Z^{2k}:0<n_i\leq N,\frac{1}{n_1}+\ldots+\frac{1}{n_k}=\frac{1}{n_{k+1}}+\ldots+\frac{1}{n_{2k}}\}\ll_{k,\epsilon} N^{k+\epsilon}.$$
\end{lem}
\begin{proof}
Suppose that 
$$\frac{1}{n_1}+\ldots+\frac{1}{n_k}=\frac{1}{n_{k+1}}+\ldots+\frac{1}{n_{2k}}$$
with $0<n_i\leq N$ for all $i$.  Let $P=\prod n_i$.  We clearly have 
$0<P<N^{2k}$.  

Suppose that a prime $p$ divides $n_i$ for some $i$.  When we multiply
the above equation by $P$, to remove the denominators, all the terms
except one will contain a factor $n_i$ and so will be divisible by $p$.  In order to have equality the exceptional term must also be divisible by $p$. Hence, since $p$ is prime, there must be an index $j\ne i$ with $p|n_j$.  It follows that if $p|P$ for a prime $p$ then 
$p^2|P$, whence $P$ is square-full.  

The number of square-full integers up to $x$ is $O(\sqrt{x})$, see Golomb \cite{golomb} for example.  Thus there are $O(N^k)$ possible values for $P$.  For each such value of $P$ the number of solutions to the equation is bounded by the divisor function $\tau_{2k}(P)$.  The result follows on using the standard estimate for these divisor functions.
\end{proof}

Now let $J^{(k)}_M(q)$ denote the number of solutions to the congruence 
\begin{equation}\label{cong}
\overline m_1+\ldots+\overline m_k\equiv \overline m_{k+1}+\ldots+\overline m_{2k}\pmod q
\end{equation}
with $1\leq m_i\leq M$ and $(m_i,q)=1$. The following generalises
Fouvry and Shparlinski's result, \cite[Lemma 3]{foushpar}. 

\begin{lem}\label{cayleycong}
Fix some $k\in\N$.  For any $\epsilon>0$ and any $M\geq 1$ we have 
$$\sum_{q\sim Q}J^{(k)}_M(q)\ll_{k,\epsilon} (QM^k+M^{2k})M^\epsilon.$$ 
\end{lem} 
\begin{proof}
For each $(m_1,\ldots,m_{2k})$ with $1\leq m_i\leq M$ we count the number of $q\sim Q$ with $(q,m_i)=1$ for which the congruence (\ref{cong}) holds.  If 
$$\frac{1}{m_1}+\ldots+\frac{1}{m_k}=\frac{1}{m_{k+1}}+\ldots+\frac{1}{m_{2k}}$$
then the congruence is satisfied for every $q\sim Q$ for which
$q$ is coprime to $\prod m_i$.  Using Lemma \ref{cayley} it follows that the contribution 
from such $2k$-tuples $(m_1,\ldots,m_{2k})$ is
$$\ll_{\epsilon,k}QM^{k+\epsilon}.$$ 
In the alternative case we define 
$$F=\prod_{i=1}^{2k}m_i(\sum_{i=1}^k m_i^{-1}-\sum_{i=k+1}^{2k}m_i^{-1})$$
so that $F$ is a non-zero integer with $|F|\ll M^{2k}$.  Since $q|F$ there are $O(M^\epsilon)$ possible values for $q$.  Thus the contribution from such $2k$-tuples $(m_1,\ldots,m_{2k})$ is 
$$\ll_{\epsilon,k}M^{2k+\epsilon}$$
so the result follows.
\end{proof}

\section{Estimates for Bilinear Sums}

Throughout this section let $\alpha_l,\beta_m$ be arbitrary complex
numbers bounded by $1$.  We will prove  estimates for sums
$$W=W_{a,q}=\twosum{l\sim L,m\sim M}{(ml,q)=1}\alpha_l\beta_me(\frac{a\overline{lm}}{q}),$$ 
either individually or on average over $q\sim Q$.  If $\beta_m=1$ then
we call $W$ a Type I sum, if not then it is Type II. 

Firstly we use Lemma \ref{cayleycong} to estimate Type II sums on average.  This is a generalisation of a bound of Fouvry and Shparlinski, \cite[Corollary 5]{foushpar}.

\begin{lem}\label{typeiimax}
Suppose that $1\leq L,M\leq Q$.  For any integer $k\geq 1$ and any $\epsilon>0$ we have 
$$\sum_{q\sim Q}\max_{(a,q)=1}|W_{a,q}|\ll_{\epsilon,k} Q(Q^{\frac{1}{2k}}L^{\frac{2k-1}{2k}}M^{\frac{1}{2}}+L^{\frac{2k-1}{2k}}M)Q^\epsilon.$$
\end{lem} 

\begin{proof}
By H\"older's inequality we have 
$$W^{2k}\leq L^{2k-1}\twosum{l\sim L}{(l,q)=1}\left|\twosum{m\sim M}{(m,q)=1}\beta_m e(\frac{a\overline{lm}}{q})\right|^{2k}.$$
Since $L<Q$ we may bound this by extending the sum over $l$ to a sum over all residues modulo $q$:
$$W^{2k}\leq L^{2k-1}\sum_{l=1}^q\left|\twosum{m\sim M}{(m,q)=1}\beta_m e(\frac{al\overline m}{q})\right|^{2k}.$$
Expanding, reordering the summation and using the orthogonality of additive characters then results in 
$$W^{2k}\ll  L^{2k-1}Q J^{(k)}_M(q).$$
Using H\"older's inequality and Lemma \ref{cayleycong} we then get 
\begin{eqnarray*}
\sum_{q\sim Q}\max_{(a,q)=1}|W|&\ll&L^{\frac{2k-1}{2k}}Q^{\frac{1}{2k}}\sum_{q\sim Q} J^{(k)}_M(q)^{\frac{1}{2k}}\\
&\leq&L^{\frac{2k-1}{2k}}Q\left(\sum_{q\sim
  Q}J^{(k)}_M(q)\right)^{\frac{1}{2k}}\\
&\ll_{\epsilon,k}&L^{\frac{2k-1}{2k}}Q(QM^k+M^{2k})^{\frac{1}{2k}}M^\epsilon\\
&\ll&Q(Q^{\frac{1}{2k}}L^{\frac{2k-1}{2k}}M^{\frac12}+L^{\frac{2k-1}{2k}}M)Q^\epsilon.\\
\end{eqnarray*}
\end{proof}

If we remove the maximum over $a$ then we can obtain a sharper estimate by exploiting the sum over $q$.

\begin{lem}\label{typeiinomax}
For any integer $a>0$, any $L,M,Q\geq 1$, and any  $\epsilon>0$, we have 
$$\sum_{q\sim Q}|W_{a,q}|\ll_\epsilon (1+\frac{a}{LMQ})^{\frac12}(QLM^{\frac12}+Q^{\frac12}L^{\frac54}M^{\frac32})(aQ)^\epsilon.$$
\end{lem}
\begin{proof}
We first consider the case when $\alpha_l,\beta_m$ are supported on integers coprime to $a$.  We trivially have 
$$\sum_{q\sim Q}|W_{a,q}|\leq \twosum{l\sim L}{(l,a)=1}W_1(l)$$
where 
$$W_1(l)=\twosum{q\sim Q}{(q,l)=1}\left|\twosum{m\sim M}{(m,aq)=1}\beta_m e(\frac{a\overline{lm}}{q})\right|.$$
By Cauchy's inequality we get 
$$W_1^2\leq Q\twosum{q\sim Q}{(q,l)=1}\left|\twosum{m\sim M}{(m,aq)=1}\beta_m e(\frac{a\overline{lm}}{q})\right|^2.$$
Expanding and reordering the summation then gives us the bound 
$$W_1^2\leq Q\twosum{m_1,m_2\sim M}{(m_1m_2,a)=1}\left|\twosum{q\sim Q}{(q,lm_1m_2)=1}e(\frac{a(m_1-m_2)\overline{lm_1m_2}}{q})\right|.$$
We can write 
$$\overline{lm_1m_2}=\frac{1-q\overline q}{lm_1m_2}$$
where $\overline q$ is an inverse of $q$ modulo $lm_1m_2$.  Therefore
$$W_1^2\leq Q\twosum{m_1,m_2\sim M}{(m_1m_2,a)=1}\left|\twosum{q\sim
  Q}{(q,lm_1m_2)=1}e\left(\frac{a(m_1-m_2)}{qlm_1m_2}\right)e\left(-\frac{a(m_1-m_2)\overline q}{lm_1m_2}\right)\right|.$$ 
If we let 
$$f(t)=e(\frac{a(m_1-m_2)}{lm_1m_2t})$$
then the factor $f(q)$ can be removed using summation by parts.  For $t\sim Q$ we have 
$$f'(t)\ll \frac{a}{LMQ^2}$$
and thus 
$$W_1^2\ll Q(1+\frac{a}{LMQ})\twosum{m_1,m_2\sim M}{(m_1m_2,a)=1}\max_{Q'\sim Q}\left|\twosum{Q\leq q<Q'}{(q,lm_1m_2)=1}e(\frac{a(m_1-m_2)\overline q}{lm_1m_2})\right|.$$
We get a contribution to this from pairs $m_1=m_2$ which is bounded by 
$$Q(1+\frac{a}{LMQ})MQ.$$
For the remaining terms let 
$$b=a(m_1-m_2)$$
and 
$$c=lm_1m_2$$
so that the inner sum is 
$$\twosum{Q\leq q<Q'}{(q,c)=1}e(\frac{b\overline q}{c}).$$
We may bound this using Lemma \ref{weil} by 
$$\ll_\epsilon ((b,c)(\frac{Q}{LM^2}+1)+(LM^2)^{\frac12})(LM^2)^\epsilon.$$
The result would be trivial if $LM^2\geq Q^2$.  We thus assume 
that $LM^2\leq Q^2$, which allows us to replace $(LM^2)^\epsilon$ by $Q^\epsilon$ in our bound. 

The contribution to our estimate for $W_1^2$ from the term $(LM^2)^{\frac12}$ is then 
$$\ll_\epsilon Q(1+\frac{a}{LMQ})L^{\frac12}M^3 Q^\epsilon$$
and that from the remaining terms is 
$$\ll_\epsilon Q(1+\frac{a}{LMQ})(\frac{Q}{LM^2}+1)Q^\epsilon\twosum{m_1,m_2\sim M}{m_1\ne m_2}(m_1-m_2,lm_1m_2),$$
where we have used that $(m_1m_2l,a)=1$.

If we write $h=m_1-m_2\ne 0$ then 
\begin{eqnarray*}
\twosum{m_1,m_2\sim M}{m_1\ne m_2}(m_1-m_2,lm_1m_2)&\ll&\sum_{m_1\sim M}\sum_{0<h\ll M}(h,lm_1(m_1+h))\\
&=&\sum_{m_1\sim M}\sum_{0<h\ll M}(h,lm_1^2)\\
&\ll_\epsilon &M^2Q^\epsilon,\\
\end{eqnarray*}
since one has in general that
$$\sum_{h=1}^H(h,n)\ll H\tau(n)\ll_\epsilon Hn^{\epsilon}$$
for any $n\in\N$.

We conclude that 
$$W_1^2\ll_\epsilon Q(1+\frac{a}{LMQ})(QM+L^{\frac12}M^3+\frac{Q}{L}+M^2)Q^\epsilon.$$
Since $L,M\geq 1$ this simplifies to 
$$W_1^2\ll_\epsilon Q(1+\frac{a}{LMQ})(QM+L^{\frac12}M^3)Q^\epsilon$$
and therefore 
$$W\ll_\epsilon (1+\frac{a}{LMQ})^{\frac12}(QLM^{\frac12}+Q^{\frac12}L^{\frac54}M^{\frac32})Q^\epsilon.$$
This completes the proof under the assumption that the coefficients
are supported on integers coprime to $a$.  To remove this assumption we begin by writing $(l,a)=u$, $a=bu$ and $l=ku$ to get 
\begin{eqnarray*}
W_{a,q}&=&\twosum{l\sim L,m\sim M}{(ml,q)=1}\alpha_l\beta_me(\frac{a\overline{lm}}{q})\\
&=&\twosum{a=ub}{(u,q)=1}\twosum{k\sim L/u,m\sim M}{(mk,q)=1,(k,b)=1}\alpha_{uk}\beta_me(\frac{b\overline{km}}{q}).\\
\end{eqnarray*}
Next we set $(m,b)=v$, $m=vj$ and $b=cv$ to rewrite this as 
$$\twosum{a=uvc}{(uv,q)=1}\twosum{k\sim L/u,j\sim M/v}{(jk,q)=1,(k,vc)=1,(j,c)=1}\alpha_{uk}\beta_{vj}e(\frac{c\overline{kj}}{q}).$$
It follows that 
\begin{eqnarray*}
\sum_{q\sim Q}|W_{a,q}|&\leq &\sum_{a=uvc}\twosum{q\sim Q}{(uv,q)=1}\left|\twosum{k\sim L/u,j\sim M/v}{(jk,q)=1,(k,vc)=1,(j,c)=1}\alpha_{uk}\beta_{vj}e(\frac{c\overline{kj}}{q})\right|\\
&\leq &\sum_{a=uvc}\sum_{q\sim Q}\left|\twosum{k\sim L/u,j\sim M/v}{(jk,q)=1,(k,vc)=1,(j,c)=1}\alpha_{uk}\beta_{vj}e(\frac{c\overline{kj}}{q})\right|.\\
\end{eqnarray*}
For each factorisation $a=uvc$ the inner sum now has coefficients supported on integers coprime to $c$ so the above bound applies.  The number of factorisations is $O(a^\epsilon)$ so the bound for the general sum is the same as that for the sum with coprimality conditions except for an additional factor $a^\epsilon$.
\end{proof}

Finally we use Lemma \ref{weil} directly, to estimate Type I sums when $L$ is small.

\begin{lem}
Suppose that $\beta_m=1$.  Then, for any $L,M\geq 1$ 
and any $\epsilon>0$ we have 
$$W\ll_\epsilon ((a,q)(\frac{LM}{Q}+L)+Q^{\frac12}L)Q^\epsilon.$$
\end{lem}
\begin{proof}
We have 
$$W\leq \twosum{l\sim L}{(l,q)=1}\left|\twosum{m\sim M}{(m,q)=1}e(\frac{a\overline{lm}}{q})\right|.$$
The result follows on applying Lemma \ref{weil} to the inner sum.
\end{proof}

Summing this result over $q\sim Q$ we immediately get the following.

\begin{lem}\label{typei}
Suppose that $L,M,Q\geq 1$ 
and that $\beta_m=1$.  For any $\epsilon>0$ we have 
$$\sum_{q\sim Q}\max_{(a,q)=1}|W_{a,q}|\ll_\epsilon (LM+Q^{\frac32}L)Q^\epsilon.$$
In addition if $a>0$ then we have 
$$\sum_{q\sim Q}|W_{a,q}|\ll_\epsilon (LM+Q^{\frac32}L)(aQ)^\epsilon.$$
\end{lem}

\section{Proof of the Theorems} 

In the sums $S_q(a;x)$ we replace the indicator function of the primes with the von Mangoldt function $\Lambda(n)$.  The contribution of prime powers $p^\alpha$ with $\alpha>1$ is $\ll_\epsilon x^{\frac12+\epsilon}$.  This is smaller than any of the bounds we will establish so it may be ignored.  In addition the factor $\log p$ may be removed using partial summation with the cost of a factor $x^\epsilon\ll Q^\epsilon$.   It is thus sufficient to establish the theorems for the sums containing $\Lambda$.

We decompose $\Lambda(n)$ using Vaughan's Identity, as described by
Davenport in \cite[Chapter 24]{dav}.  We will use $U=V\leq
x^{\frac13}$; the precise choice of $U$ for each theorem will be given
later.  The sum $S_q(a;x)$ is decomposed into Type I and II sums with
$LM\asymp x$.  The precise forms of the sums are given by Fouvry and
Shparlinski in \cite{foushpar}.  The coefficients are not all bounded
by $1$ but they are bounded by a divisor function.  This divisor
function may be absorbed into the $Q^\epsilon$ term.  We must estimate
Type I sums for $L\leq U^2$ and Type II sums for $U\leq L\leq x/U$.
 Since $U^2\le x/U$ any Type I sum with $U\le L\le L^2$ may be
regarded as a Type II sum.  Hence it will be enough to consider Type
I sums with $L\leq U$ and Type II sums with $U\leq L\leq x/U$.  The variables of summation are restricted by the condition $lm\sim x$.
In the Type II sums this may be removed by Fourier analysis, see for
example the start of Garaev's proof, \cite[Lemma 2.4]{garaev}. 
 For the Type I sums, if we are simply treating them as Type II sums
 then the same argument applies, whereas if we are using Lemma
 \ref{typei} then it is clear that a condition $lm\sim x$ can be
 introduced by modifying the proof. 

\subsection{Proof of Theorem \ref{thm1}}

The sums arising from Vaughan's identity are of the form 
$$\sum_{q\sim Q}\max_{(a,q)=1}|W_{a,q}|.$$
We estimate the Type II sums using Lemma \ref{typeiimax}.  We have 
$$L,M\leq \frac{x}{U}$$
and thus the hypotheses are satisfied provided that our choice of $U$ satisfies $\frac xU\leq Q$.  Recalling that $M\ll x/L$ the bound from Lemma \ref{typeiimax} is 
$$Q(Q^{\frac{1}{2k}}x^{\frac12}L^{\frac{k-1}{2k}}+xL^{\frac{-1}{2k}})Q^\epsilon.$$
For $x^{\frac25}\leq L\leq x^{\frac12}$ we apply this with $k=2$ to get a bound of 
$$Q(Q^{\frac{1}{4}}x^{\frac58}+x^{\frac{9}{10}})Q^\epsilon.$$
For $x^{\frac35}\leq L\leq x/U$ we use  $k=3$ to get 
$$Q(Q^{\frac{1}{6}}x^{\frac56}U^{-\frac13}+x^{\frac{9}{10}})Q^\epsilon.$$
Since $LM\asymp x$ and we can interchange $l,m$ in our sums these two bounds in fact cover the whole range $U\leq L\leq x/U$.  The contribution of all our Type II sums is therefore 
$$\ll_\epsilon Q(Q^{\frac{1}{4}}x^{\frac58}+x^{\frac{9}{10}}+Q^{\frac16}x^{\frac56}U^{-\frac13})Q^\epsilon.$$

We need to estimate Type I sums for $L\leq U$.  Lemma \ref{typei} gives a bound for these sums of 
$$(x+Q^{\frac32}U)Q^\epsilon.$$
Since $x\leq Q^{\frac32}$ and $U\geq 1$ the second term is larger and thus 
$$\sum_{q\sim Q}\max_{(a,q)=1}|S_q(a;x)|\ll_\epsilon Q(Q^{\frac{1}{4}}x^{\frac58}+x^{\frac{9}{10}}+Q^{\frac16}x^{\frac56}U^{-\frac13}+Q^{\frac12}U)Q^\epsilon.$$
We now choose 
$$U=\min(x^{\frac13},Q^{-\frac14}x^{\frac58}).$$
Since $Q^{\frac23}\leq x\leq Q^{\frac32}$ we have 
$$1\leq U\leq x^{\frac13}$$
and 
$$\frac{x}{U}\leq Q.$$
This choice of $U$ is therefore admissible so we can conclude that 
$$\sum_{q\sim Q}\max_{(a,q)=1}|S_q(a;x)|\ll_\epsilon Q(Q^{\frac{1}{4}}x^{\frac58}+x^{\frac{9}{10}}+Q^{\frac16}x^{\frac{13}{18}})Q^\epsilon,$$
thus proving Theorem \ref{thm1}.

It should be noted that we have not used the full power of Lemma \ref{typeiimax} and thus our result could be improved for certain sizes of $x$.  In particular as $x$ decreases it becomes beneficial to use larger  values of $k$ in the lemma.  Theorem \ref{thm1} gives a nontrivial estimate provided that $x\geq Q^{\frac23+\delta}$ for some $\delta>0$.  By using larger values of $k$ it should be possible to get a result for $x\geq Q^{\frac12+\delta}$.  When $x$ is a fixed power of $Q$ it is not difficult to enumerate all the different bounds coming from Lemma \ref{typeiimax} and work out what the largest term is.  However, it seems that the optimal result for all $x$ would be very awkward to write down and prove.

\subsection{Proof of Theorem \ref{ternary}}

It follows from Theorem \ref{thm1} that for every $\epsilon>0$ there exists a $\delta>0$ such that 
\begin{equation}\label{ternarybound}
\sum_{q\sim Q}\max_{(a,q)=1}|S_q(a;x)|\ll x^{2-\delta}
\end{equation}
for $x\leq Q\leq x^{\frac{23}{21}-\epsilon}$.  In general, if we have a constatnt $\alpha\in (1,2)$ such that the estimate (\ref{ternarybound}) holds for $x\leq Q\leq x^\alpha$ then Baker's analysis, \cite{baker}, shows that we may  replace $1.167\ldots$ by the solution, $\theta$, of the equation 
$$2\theta-\alpha-2+2(2-\alpha)\log\left(\frac{\theta+\alpha-2}{2\alpha-2}\right)=0.$$
When $\alpha=\frac{23}{21}$ this equation becomes 
$$42\theta-65+38\log\left(\frac{21\theta-19}{4}\right)=0.$$
The root of this is $\theta=1.188\ldots$, as given in Theorem \ref{ternary}.

\subsection{Proof of Theorem \ref{thm2}}

The sums arising from Vaughan's identity are now of the form 
$$\sum_{q\sim Q}|W_{a,q}|.$$
We estimate the Type II sums using Lemma \ref{typeiinomax}.  Since $LM\asymp x$ this gives a bound 
$$(1+\frac{a}{xQ})^{\frac12}(Qx^{\frac12}L^{\frac12}+Q^{\frac12}x^{\frac32}L^{-\frac14})(aQ)^\epsilon.$$
It is sufficient to estimate the Type II sums for $x^{\frac12}\leq L\leq x/U$.  In this range we get a bound of 
$$(1+\frac{a}{xQ})^{\frac12}(QxU^{-\frac12}+Q^{\frac12}x^{\frac{11}{8}})(aQ)^\epsilon.$$
Lemma \ref{typei} gives us a bound for the Type I sums with $L\leq U$ of $Q^{\frac32+\epsilon}Ua^\epsilon$.  We now choose 
$$U=\min(x^{\frac13},Q^{-\frac13}x^{\frac23}).$$
Since $x\geq Q^{\frac12}$ we have 
$$1\leq U\leq x^{\frac13}.$$
This choice of $U$ is therefore admissible and we get 
$$\sum_{q\sim Q}|S_q(a;x)|\ll_\epsilon (1+\frac{a}{xQ})^{\frac12}(Qx^{\frac56}+Q^{\frac76}x^{\frac23}+Q^{\frac12}x^{\frac{11}{8}})(aQ)^\epsilon.$$
If $x\leq  Q$ then 
$$Qx^{\frac56}\leq Q^{\frac76}x^{\frac23}$$
and if $x\geq Q$ then 
$$Qx^{\frac56}\leq Q^{\frac12}x^{\frac{11}{8}}.$$
The term $Qx^{\frac56}$ is therefore not necessary in our bound so Theorem \ref{thm2} follows.

\newpage
\bibliographystyle{plain}
\bibliography{biblio} 

\bigskip
\bigskip

Mathematical Institute,

24--29, St. Giles',

Oxford

OX1 3LB

UK
\bigskip

{\tt irving@maths.ox.ac.uk}

\end{document}